\documentclass[preprint,twoside, 11pt]{elsarticle}
\usepackage{amssymb}
\usepackage{amsmath}
\usepackage{amsthm}
\usepackage{graphicx}
\usepackage{epstopdf}
\usepackage{amsfonts}
\usepackage{color}
 
\newtheorem{thm}{Theorem}[section]

\newtheorem{lem}[thm]{Lemma}

\newtheorem{rem}[thm]{Remark}
\newtheorem{exm}[thm]{Example}
\newtheorem{assu}[thm]{Assumption}

\numberwithin{equation}{section}

  \newcommand{\sfrac}[2]{#1/(#2)}
 
\newcommand{\abs}[1]{\left\vert#1\right\vert}

\newcommand{\mean}[1]{\mathbb{E}\lbrack #1\rbrack}
 
 \allowdisplaybreaks[1]
\usepackage{etoolbox}
\patchcmd{\pprintMaketitle}
  {\fi\hrule}
  {\fi\ifvoid\extrainfobox\else\unvbox\extrainfobox\par\vskip10pt\fi\hrule}
  {}{}

\newenvironment{extrainfo}
  {\global\setbox\extrainfobox=\vbox\bgroup\parindent=0pt }
  {\egroup}
\newsavebox\extrainfobox

\begin{document}
\begin{frontmatter} 
\title{Order-preserving  strong schemes for  SDEs with locally Lipschitz
coefficients}
\author{Zhongqiang  Zhang $^{a,*}$ and Heping Ma $^{b}$}
\address{$^{a}$ {Department of Mathematical Sciences, Worcester Polytechnic Institute, Worcester, MA, 01609 (zzhang7@wpi.edu).  } Corresponding author. }
\address{$^b$ {Department of Mathematics, Shanghai University, Shanghai, 200444 (hpma@shu.edu.cn).}}

\begin{abstract}
We introduce a class of  explicit balanced schemes for stochastic differential equations with coefficients of superlinearly growth
satisfying a global monotone condition.  The first scheme is  a balanced Euler scheme and
is  of order half   in the mean-square sense whereas it is of order one  under additive noise. The second scheme is   a balanced Milstein scheme, which  is of order one in the mean-square sense.  Some numerical results are presented.
\end{abstract}

\begin{extrainfo}
\text{Highlights}
\begin{itemize}
\item Explicit schemes are proposed for SDEs with no globally Lipschitz coefficients. 
\item  Convergence rates are proved to be half-order and first-order. 
\item Numerical comparison with  reliable numerical schemes is made. 
\end{itemize}
\end{extrainfo}
 \date{\today}

\end{frontmatter} 
\maketitle  

\section{Introduction}

Let $(\Omega,\mathcal{F},P)$ be a complete probability space and
$\mathcal{F}_{t}^{w}$ be an increasing family of $\sigma$-subalgebras of
$\mathcal{F}$ induced by $w(t)$ for $0\leq t\leq T$, where $(w(t),\mathcal{F}
_{t}^{w})=((w_{1}(t),\ldots,w_{m}(t))^{\top},\mathcal{F}_{t}^{w})$ is an $m
$-dimensional standard Wiener process. We consider  numerical methods for the system of Ito
stochastic differential equations (SDE):
\begin{equation} \label{eq:sde-ito-vec} 
dX=a(t,X)dt+\sum_{r=1}^{m}\sigma_{r}(t,X)dw_{r}(t),\ \ t\in(t_{0} 
,T],\ X(t_{0})=X_{0},
\end{equation}
where $X,$\ $a,$\ $\sigma_{r}$ are $d$-dimensional column-vectors and $X_{0} $
is independent of $w$. We suppose that any solution $X_{t_{0},X_{0}}(t)$ of
\eqref{eq:sde-ito-vec} is  well-defined on $[t_{0},T]$.

 In this work,  we  propose   two  balanced  explicit schemes with equispaced time steps sizes for  \eqref{eq:sde-ito-vec}  when the coefficients $a(t,x)$ and $\sigma_r(t,x)$   satisfy no globally Lipschitz  conditions.   When the global Lipschitz conditions are violated,   many existing explicit numerical schemes for SDEs with Lipschitz coefficients are not stable and thus not convergent any more, see e.g. \cite{KloPla-B92,Mil-B95,MilTre-B04}.   
  The  forward Euler scheme  for  the equation $dX= -X^3\,dt +dw$ fails to converge  in the moments and mean-square sense,  
   see e.g. \cite{HutJK11,MilTre05}.  
  More examples have been discussed in  e.g. \cite{HigMS02,Hu96,MilTre05,Tal02}.   The failure of the   forward Euler scheme    has motivated many research on numerical  methods for \eqref{eq:sde-ito-vec} with non-globally Lipschitz conditions, see a   literature review on this
   topic  in \cite{HutJen15}.  For strong schemes for SDEs with non-globally
 Lipschitz coefficients of superlinear growth,
 several types of methods have been proposed:
 \begin{itemize}
 \item tamed explicit schemes, see e.g.    \cite{HutJen14,HutJK11,HutJK12,HutJW13, Sab13,Sab13a,Szp13,TreZhang13,WangGan13b,ZongWH14},
  where the coefficients are    approximated by the function of the form $\sfrac{f(x)}{1+h^{\alpha}\abs{f(x)}}$ ($0\leq \alpha\leq 1$) to control their superlinear growth.
 \item explicit truncation schemes,  where    the coefficients are set to be zero when solutions  become  very large, see e.g. \cite{LiuMao13,Mao15,MilTre05}.
  \item   implicit schemes, where the drift and/or diffusion coefficients are treated implicitly, see e.g.  \cite{HigMS02,Hu96,MaoSzp13,MaoSzp13a,TreZhang13}.
 \end{itemize}

However, for no globally Lipschitz coefficients,  i.e.,  neither drift nor diffusion coefficients are Lipschitz, no high-order schemes have been proposed, i.e.
 all the schemes proposed are of order half, see e.g. \cite{HutJen14,HutJW13,Sab13,TreZhang13,Szp13}.   To the best of our knowledge,  the  first-order scheme  was  only discussed in \cite{WangGan13b}  for SDEs with locally Lipschitz drift coefficients but Lipschitz diffusion coefficients.  
 Moreover, these schemes are not of strong order one under additive noise while we recall that     classical half-order schemes like the  Euler  scheme usually become first-order schemes  under additive noise when SDEs have Lipschitz continuous coefficients.  This motivates us to obtain high-order schemes. 
 
 As we are considering balanced explicit schemes, let us first   review briefly balanced explicit schemes for  SDEs with non-globally Lipschitz coefficients.
  When   only  the coefficient $a(t,x)$ violates  the Lipschitz condition and   satisfies one-sided Lipschitz condition
 (or monotone condition) and grows superlinearly,  some  explicit schemes (tamed schemes, one type of balanced schemes \cite{MilPS98}) have been proposed for  SDEs under such conditions, see e.g.
 tamed Euler schemes \cite{HutJK12,Sab13a}, tamed  Milstein scheme \cite{WangGan13b}. Compared with the classical Euler scheme and
 Milstein scheme, these schemes  have an approximate drift term  $a(t_k,X_k)/(1+h^{\alpha}\abs{a(t_k,X_k)})$ (or in a similar form) instead of  the drift terms $a(t_k,X_k)$ to control the growth of the drift, where
 $\alpha=1$ in \cite{HutJK12,WangGan13b} and $\alpha=1/2$ in \cite{Sab13}.

 When   the coefficients, both $a(t,x)$ and $\sigma_r(t,x)$,  violate  the Lipschitz conditions,  the aforementioned tamed schemes also  fails to converge in the mean-square sense.
 In this case, Ref. \cite{HutJen15} proposed a ``fully tamed'' Euler scheme
 \begin{equation}    
 X_{k+1}=X_k +\frac{a(t_k,X_k)h+\sum_{r=1}^m \sigma_r(t_k,X_k)\xi_{rk}\sqrt{h}}{\max(1, h\abs{a(t_k,X_k)h+\sum_{r=1}^m  \sigma_r(t_k, X_k)\xi_{rk}\sqrt{h} } }.
 \end{equation}
 This scheme is proved to converge without convergence order in \cite{HutJen15}. However, it is shown that the solutions from this scheme  become  oscillatory at certain values after
 the term $h\abs{ a(t_k,X_k)h+\sum_{r=1}^m \sigma_r(t_k, X_k)\xi_{rk}\sqrt{h} } $ is larger than one.    
 Under a global monotone condition and  some polynomials  growth conditions,  Ref. \cite{TreZhang13} proposed  the following balanced  scheme (tamed scheme)
 \begin{equation}\label{eq:tamed-TreZhang}
 X_{k+1}=X_k +\frac{a(t_k,X_k)h+\sum_{r=1}^m \sigma_r(t_k,X_k)\xi_{rk}\sqrt{h}}{1+ h\abs{a(t_k,X_k)h}+\sum_{r=1}^m\abs{ \sigma_r(t_k,X_k)\xi_{rk} }\sqrt{h}  },
 \end{equation}
 and proved a half-order convergence of this scheme. They showed  that the scheme is still of order half for additive noise.  Ref. \cite{Sab13} pointed out that
 the scheme \eqref{eq:tamed-TreZhang} is not applicable for some critical situations where the solution to \eqref{eq:sde-ito-vec} has only a finite number of moments.
 The author then  proposed the following   scheme
 \begin{equation}\label{eq:tamed-Sab}
 X_{k+1}=X_k +\frac{a(t_k,X_k)h+\sum_{r=1}^m \sigma_r(t_k, X_k)\xi_{rk}\sqrt{h}}{1+  \abs{a(t_k,X_k)}h^{\beta}+\sum_{r=1}^m\abs{ \sigma_r(t_k,X_k)} h^{\beta} },
 \end{equation}
 where the   scheme was proved to converge in the mean-square sense with order  half  when $\beta=1/2$.
 A general tamed scheme of this type (with drift and diffusion coefficients divided by some functional of coefficients plus one)  is  proposed in
 \cite{Szp13}  to accommodate  Lyapunov stability  of  SDEs  rather than  to simply focus on  $L^p$-stability which may not be available for some SDEs. 
  Under  general conditions, Refs. \cite{HutJen14,HutJW13} proposed   a tamed Euler scheme of a similar type for SDEs with exponential moments
  and proved stability and  half-order convergence in the $L^p$ sense.

 In tamed schemes (balanced explicit schemes), it is actually
 the use of the function $\sfrac{f(x)}{1+h^\alpha\abs{f(x)}}$ that  prevents the lifting of order of tamed schemes for SDEs (e.g.   \eqref{eq:tamed-TreZhang}  and \eqref{eq:tamed-Sab}). In the aforementioned tamed schemes, the approximation of  diffusion coefficients is of order $1$:
 \[\frac{\sigma(x)\sqrt{h}}{1+h^{1/2}\abs{\sigma(x)}} -\sigma(x)\sqrt{h}=  \frac{\sigma(x) \abs{\sigma(x)}h}{1+h^{1/2}\abs{\sigma(x)}} .\]
 According to Theorem \ref{thm:ms-convergence-local-global} (the fundamental theorem of strong convergence, see also \cite{TreZhang13}), we then can not have a strong scheme  of order   higher than half since $q_2$ in \eqref{eq:one-step-local-ms}  is not more than $1$. Here we  propose  the following scheme
 \begin{equation}\label{eq:sine-tamed}   
 X_{k+1}=X_{k}+\mathcal{T} (a(t_{k},X_{k})h)+\mathcal{T}(\sum_{r=1}^{m}\sigma_{r}(t_{k},X_{k})\xi_{rk}\sqrt{h}),
 \end{equation}
  where $\mathcal{T}(\cdot)$ is  either the  hyperbolic tangent function  or the sine function. 
 This scheme will be proved to be of half-order mean-square convergence  in general and is of first-order  mean-square convergence for additive noise.   Moreover, we can obtain a first-order strong scheme
  \begin{eqnarray} \label{eq:sine-tamed-milstein}   
     X_{k+1}  &=&  X_{k}+ \mathcal{T}(a(t_k,X_{k}) h )+\mathcal{T}(\sum_{r=1}^{m}\sigma_{r}(t_{k},X_{k})  \xi_{rk}\sqrt{h} )  + \mathcal{T}\left(\sum_{i,r=1}^m\Lambda_i \sigma_r(t,X_k) I_{i,r,t_k} \right),   
  \end{eqnarray}   where
  $ \displaystyle I_{i,r,t_k} = \int_{t_k} ^{t_{k+1}} \int_{t_k} ^{s} \,dw_i\,dw_r$.
Although  it is expensive to simulate the Levy area $I_{i,r,t_k}$ and thus the Milstein scheme,  we have significant reduction in the amount of operations for
commutative noises,
i.e.
\begin{equation}\label{eq:commutative-condition}
 \Lambda_i \sigma_r = \Lambda _r \sigma_i, \quad   \Lambda_r= \sigma_r^\top \frac{\partial}{\partial x}.
\end{equation}  In this case,  we  can use   increments of Brownian motions instead of  double Ito integrals $I_{i,r,t_k}$ in  \eqref{eq:sine-tamed-milstein}
 since $I_{i,r,t_k}+I_{r,i,t_k}=(\xi_{ik}\xi_{rk}-\delta_{ir})h/2$ where $\delta_{ir}$ is the Kronecker delta function. We then   simplify    \eqref{eq:sine-tamed-milstein}  as
  \begin{eqnarray} \label{eq:sine-tamed-milstein-com}
    X_{k+1}  & = &   X_{k}+ \mathcal{T}(a(t_k,X_{k}) h )+\mathcal{T}(\sum_{r=1}^{m}\sigma_{r}(t_{k},X_{k})  \xi_{rk}\sqrt{h} )  \notag \\
                    &  &       + \mathcal{T}\left(\frac{1}{2}\sum_{i,r=1}^m\Lambda_i \sigma_r(t,X_k)(\xi_{ik}\xi_{rk}-\delta_{ir})h \right).
 \end{eqnarray}   
In Section \ref{sec:proofs}, we will prove the convergence orders when the sine function is used and the proofs for the case of $\tanh$ are similar.  

We remark   that the use of the hyperbolic tangent or the sine function is motived   by obtaining higher-order schemes. 
In many applications, order-preserving may not be enough, e.g. in long-time simulation and when solutions are positive. One may want to preserve certain structure  of solutions in numerical schemes. We expect that it is possible to design a structure preserving scheme (half- or first- order) using different tame functions rather than  the hyperbolic tangent or the sine function.  For example,  a positivity-preserving scheme  using  the tame functions 
	$\mathcal{T}(x) =x /(1+h\abs{x})$ and the absolute function 
	  is proposed in \cite{BiZhang16} for a nonlinear SDE  with locally Lipschitz drift coefficient and  H\"older  continuous diffusion coefficient.

 In the next section, we present our requirements on coefficients in \eqref{eq:sde-ito-vec} and the fundamental  theorem  of strong convergence which are necessary for our proofs of convergence orders of the two schemes
 \eqref{eq:sine-tamed} and \eqref{eq:sine-tamed-milstein}.
We present the proof of half-order convergence of \eqref{eq:sine-tamed} in Section 3 and  that of first-order convergence of
\eqref{eq:sine-tamed-milstein} in Section 4. We present some numerical results in Section 5 at the end of the paper.

\section{Preliminary \label{sec:theo}}

  Throughout the paper,  we use the letter $K$ to denote
generic constants which are independent of $h$ (time step size) and $k$ (time steps).

Let $X_{t_{0},X_{0}}(t)=X(t),$\ $t_{0}\leq t\leq T,$ be a solution of the
system \eqref{eq:sde-ito-vec}. We will assume the bounded moments of initial condition,
global monotone condition and local Lipschitz condition as  follows:
\begin{assu}\label{asup:one-side-lip}
 (i) The initial condition is such that%
\begin{equation}
\mathbb{E}|X_{0}|^{2p}\leq K<\infty,\ \ \text{for all \ } p\geq1. \label{mom0}%
\end{equation}
(ii) For a sufficiently large $p_{0}\geq1$ there is a constant $c_{1}\geq0$
such that for $t\in\lbrack t_{0},T]$,
\begin{equation} \label{eq:monotone-global}%
(x-y,a(t,x)-a(t,y))+\frac{2p_{0}-1}{2}\sum_{r=1}^{m}|\sigma_{r}(t,x)-\sigma
_{r}(t,y)|^{2}\leq c_{1}|x-y|^{2},\ x,y\in\mathbb{R}^{d}.
\end{equation}
(iii) There exist $c_{2}\geq0$ and $\varkappa\geq1$ such that for $t\in\lbrack
t_{0},T]$,
\begin{equation}\label{eq:local-lip}%
|a(t,x)-a(t,y)|^{2}\leq c_{2}(1+|x|^{2\varkappa-2}+|y|^{2\varkappa
-2})|x-y|^{2},\,\ x,y\in\mathbb{R}^{d}.
\end{equation}

\end{assu}

Define   $X_{t,x}(t+h)$ of \eqref{eq:sde-ito-vec} as
\begin{equation}
X_{t,x}(t+h) =x + \int_t^{t+h} a(\theta,X_{t,x}(\theta)) \,dt + \int_t^{t+h} \sum_{r=1}^m\sigma_r(\theta,X_{t,x}(\theta)) \,dw_r,
\end{equation}
and introduce the one-step approximation $\bar{X}_{t,x}(t+h),$ $t_{0}\leq
t<t+h\leq T,$ to the solution $X_{t,x}(t+h)$
\begin{equation} \label{eq:one-step-abstract-v0}
\bar{X}_{t,x}(t+h)=x+A(t,x,h;w_{i}(\theta)-w_{i}(t),\;i=1,\ldots
,m,\;t\leq\theta\leq t+h).
\end{equation}
Using the one-step approximation \eqref{eq:one-step-abstract-v0}, we recurrently construct the
approximation $(X_{k},\mathcal{F}_{t_{k}}),\;k=0,\ldots,N,\;t_{k+1}
-t_{k}=h_{k+1},\;T_{N}=T$ with $X_0=X(t_0)$:
\begin{equation}\label{eq:one-step-abstract}
 X_{k+1}=\bar{X}_{t_{k},X_{k}}(t_{k+1}) =X_{k}+A(t_{k},X_{k},h_{k+1};w_{i}(\theta)-w_{i}(t_{k}),\;i=1,\ldots
,m,\;t_{k}\leq\theta\leq t_{k+1}).\notag
\end{equation}

For simplicity, we will consider a uniform time step size, 
i.e., $h_{k}=h $ for all $k.$

\begin{thm}[\cite{TreZhang13}] \label{thm:ms-convergence-local-global}
Suppose (i) Assumption \ref{asup:one-side-lip} holds;

(ii) The one-step approximation $\bar{X}_{t,x}(t+h)$ from  \eqref{eq:one-step-abstract-v0} has
the following orders of accuracy: for some $p\geq1$ there are $\alpha\geq1,$
$h_{0}>0,$ and $K>0$ such that for arbitrary $t_{0}\leq t\leq T-h,$%
\ $x\in\mathbb{R}^{d},$ and all $0<h\leq h_{0}:$
\begin{equation}\label{eq:one-step-local-mean}
\abs{\mean{X_{t,x}(t+h)-\bar{X}_{t,x}(t+h)}}\leq
K(1+|x|^{2\alpha})^{1/2}h^{q_{1}}\,,
\end{equation}
\begin{equation}\label{eq:one-step-local-ms}
\left[  \mathbb{E}|X_{t,x}(t+h)-\bar{X}_{t,x}(t+h)|^{2p}%
\right]  ^{1/(2p)}\leq K(1+|x|^{2\alpha p})^{1/(2p)}h^{q_{2}}\,
\end{equation}
with
\begin{equation}
q_{2}\geq\frac{1}{2}\,,\;q_{1}\geq q_{2}+\frac{1}{2}\,; \label{Ba07-lp}%
\end{equation}
(iii) The approximation $X_{k}$ from  \eqref{eq:one-step-abstract} has bounded moments, i.e.,
for some $p\geq1$ there are $\beta\geq1,$ $h_{0}>0,$ and $K>0$ such that for
all $0<h\leq h_{0}$ and all $k=0,\ldots,N$:
\begin{equation}
\mathbb{E}|X_{k}|^{2p}<K(1+\mathbb{E}|X_{0}|^{2p\beta}). \label{appmomt}%
\end{equation}
Then for any $N$ and $k=0,1,\ldots,N$ the following inequality holds:
\begin{equation}\label{eq:global-err-ms}
\left[  \mathbb{E}|X_{t_{0},X_{0}}(t_{k})-\bar{X}_{t_{0},X_{0}}(t_{k}%
)|^{2p}\right]  ^{1/(2p)}\leq K(1+\mathbb{E}|X_{0}|^{2\gamma p})^{1/(2p)}%
h^{q_{2}-1/2},
\end{equation}
where $K>0$ and $\gamma\geq1$ do not depend on $h$ and $k,$ i.e., the order of
accuracy of the method  \eqref{eq:one-step-abstract}  is $q=q_{2}-1/2.$
\end{thm}

According to this theorem, we can obtain the convergence order of a one-step method  by providing
boundedness of  moments  and local truncation error of the one-step method.
With this theorem, we will prove convergence orders of our balanced Euler and Milstein scheme in next two sections.
  %
The proof for our balanced Euler scheme will be given in details while the proof for our balanced Milstein scheme will  be briefed  with  necessary details since the idea of the proof
is very similar.

In the proofs, we will frequently use the following facts
 \begin{equation}\label{eq:polynomial-growth-deduced}
 \abs{a(t,x)}^2\leq K(1+\abs{x}^{2\kappa}),\quad \sum_{r=1}^m \abs{\sigma_r(t,x)}^2\leq K(1+\abs{x}^{\kappa+1}).
 \end{equation}
which can be readily seen from  \eqref{eq:local-lip} and  \eqref{eq:monotone-global}.
From  the global monotone condition  \eqref{eq:monotone-global}, we can readily obtain
\begin{equation}\label{eq:sde-moment-bound-global-monotone}%
\mean {\abs{X_{t_0,X_0}(t)}^{2p} }<K(1+\mean{\abs{X_0}^{2p}}),\ 1\leq p< p_{0},\ \ t\in (t_{0},T].
\end{equation}


 \section{The balanced Euler scheme} \label{sec:proofs}
 In this section, we prove a half-order mean-square convergence of our balanced Euler scheme
\eqref{eq:sine-tamed}.   For additive noise, we prove that  \eqref{eq:sine-tamed} is a first-order scheme.
By Theorem~\ref{thm:ms-convergence-local-global},  we need to prove
 boundedness of moments and local truncation error.   We consider only the case when $\mathcal{T}(\cdot)$ is the sine function as  the proof for $\mathcal{T}(\cdot)=\tanh(\cdot)$ is similar.

\subsection{Boundedness of moments of the solutions to \eqref{eq:sine-tamed}}
We will follow the recipe of the proof of moments boundedness in \cite[Section 3]{TreZhang13}, which uses a stopping time technique, see also, e.g.
\cite{HutJen15,MilTre05}.

\begin{lem}\label{lem:mom-sine-tamed}
Suppose Assumption \ref{asup:one-side-lip} holds with
sufficiently large $p_{0}$. For all natural $N$ and all $k=0,\ldots,N$ the
following inequality holds for moments of the scheme \eqref{eq:sine-tamed}
\begin{equation}\label{eq:mom-sine-tamed}
\mathbb{E}|X_{k}|^{2p}\leq K(1+\mathbb{E}|X_{0}|^{2p\beta}),\ \ 2\leq
2p<\frac{2p_{0}}{3\varkappa-3}-1,
\end{equation}
where the constants $\beta\geq1$ and $K>0$ are  independent of $h$ and $k$.
\end{lem}

\begin{proof}
We   consider only the case
$\varkappa>1$ while  the case $\varkappa=1$ (i.e., when $a(t,x)$ is globally Lipschitz) can be derived similarly.  

The key to prove the boundedness of moments is to estimate the growth of the solution under some events \begin{equation}
\tilde{\Omega}_{R,k}:=\{\omega:|X_{l}|\leq R(h),\ l=0,\ldots,k\},\ \label{event}%
\end{equation}
where $hR^{\varkappa}(h)<1$ such that
\begin{equation} \label{eq:mom-bound-nonrareevent}
  \mean{\chi_{\tilde{\Omega}_{R,k}}(\omega)\abs{X_k}^{2p} }\leq K(1+  \mean{\abs{X_0}^{2p} }).
\end{equation}
For the  compliments of  $\tilde{\Omega}_{R,k},$ denoted by $\tilde{\Lambda}_{R,k},$
we will prove the boundedness of moments starting from the following observation for \eqref{eq:sine-tamed} that
\begin{equation} \label{eq:sine-tame-bound}%
\abs{X_{k+1}}\leq\abs{X_k}  +2\leq \abs{X_0}+2(k+1).
\end{equation}

  We first prove the lemma for
integer $p\geq1.$ We have
\begin{eqnarray}\label{eq:mom-bound-cross-term-key}
  && \mean{\chi_{\tilde{\Omega}_{R,k+1}}(\omega)\abs{X_{k+1}}^{2p} } \\
   &\leq &
\mean{\chi_{\tilde{\Omega}_{R,k}}(\omega)\abs{X_{k+1}}^{2p}}   = \mean{\chi_{\tilde{\Omega}_{R,k}}(\omega)\abs{(X_{k+1}-X_{k})+X_{k} }^{2p} } \notag\\
& \leq&\mean{\chi_{\tilde{\Omega}_{R,k}}(\omega)\abs{X_k}^{2p} }%
+K\sum_{l=3}^{2p}\mathbb{E}\chi_{\tilde{\Omega}_{R,k}}(\omega)\left\vert
X_{k}\right\vert ^{2p-l}|X_{k+1}-X_{k}|^{l}+\mean{\chi_{\tilde{\Omega}_{R,k}}(\omega)\abs{X_k}^{2p-2}A},\notag
\end{eqnarray}
where $\displaystyle A =  \chi_{\tilde{\Omega}_{R,k}}(\omega) \mean{  2p(X_{k},X_{k+1}-X_{k})+p(2p-1)|X_{k+1}-X_{k}|^{2} |\mathcal{F}_{t_k}}.$

Since $\xi_{rk}$ are
independent of $\mathcal{F}_{t_{k}}$ and  the normal distribution  is symmetric, we obtain
\begin{equation}
   \chi_{\tilde{\Omega}_{R,k}}\mean{  \sum_{r=1}%
^{m}\sigma_{r}(t_{k},X_{k})\xi_{rk}\sqrt{h}|\mathcal{F}%
_{t_k}} =0,\label{tml41}
\end{equation}
and
\begin{equation}
\chi_{\tilde{\Omega}_{R,k} }
\mean{  \sum_{r=1}^m\abs{ \sigma_{r}(t_k,X_k)\xi_{rk} }^2|\mathcal{F}_{t_k} }
 =\chi_{\tilde{\Omega}_{R,k} } \sum_{r=1}^{m}\abs{\sigma_{r}(t_{k},X_k)}^2 . \label{tml42}%
\end{equation}
Similarly, we have, also by the asymmetry of the sine function,
\begin{equation}
   \chi_{\tilde{\Omega}_{R,k}}\mean{ \sin( \sum_{r=1}^{m}\sigma_{r}(t_{k},X_{k})\xi_{rk}\sqrt{h}) |\mathcal{F}  _{t_k}} =0.
\end{equation}
Then the conditional expectation in  \eqref{eq:mom-bound-cross-term-key} becomes
\begin{eqnarray*}
A  & = &2p\chi_{\tilde{\Omega}_{R,k}}\mean{(X_k, \sin(a(t_k,X_k)h)
   +\frac{2p-1}{2}\abs{X_{k+1}-X_k}^2 | \mathcal{F}_{t_k}}   \\
&=&2p\chi_{\tilde{\Omega}_{R,k}} \mean{(X_k, a(t_k,X_k)h)
    +\frac{2p-1}{2}(\abs{\sin(a(t_{k},X_{k})h)}^2+ \abs{  \sin(\sum_{r=1}^{m} \sigma_{r}(t_{k},X_{k})\xi_{rk}\sqrt{h})}^2)| \mathcal{F}_{t_k}}   \\
&&+ 2p\chi_{\tilde{\Omega}_{R,k}} \mean{(X_k,  a(t_k,X_k)h - \sin(a(t_k, X_k)h) )| \mathcal{F}_{t_k}}   \\
&\leq& 2p\chi_{\tilde{\Omega}_{R,k}}h\mean{ (X_{k},a(t_{k},X_{k})  +
 \frac{2p-1}{2} \sum_{r=1}^{m}|\sigma_{r}(t_{k},X_{k})|^{2}  |\mathcal{F}_{t_{k} }} + p(2p-1)\chi_{\tilde{\Omega}_{R,k}}\abs{a(t_{k},X_{k})}^2h^{2}\notag\\
&  &  + 2p\chi_{\tilde{\Omega}_{R,k}}  \abs{X_k}\abs{a(t_{k},X_{k})}^2h^2,
\end{eqnarray*}
where we have used the fact that $\abs{\sin(y)}\leq\abs{y}$ and the following estimate
\begin{equation}\label{eq:basic-inequality-sinetamed}
 \abs{y-\sin(y)}=\abs{(1-\cos(\theta y))y}\leq 2\abs{y}\abs{\sin(\theta y/2)}^2,\quad \text{for some}    \theta \in [0,1].
\end{equation}
 In fact, by Taylor'expansion with the remainder in Lagrange form,  there exists some  $\theta \in [0,1]$ such that $y-\sin(y) =(1 -\cos(\theta y))y$. 
Using the global monotone condition \eqref{eq:monotone-global}  and the growth condition \eqref{eq:polynomial-growth-deduced},  we obtain%
\begin{align}\label{eq:mom-bound-cross-conditionalexp}
A  &  \leq K\chi_{\tilde{\Omega}_{R,k}} (h+\abs{X_k}^{2}h + \abs{X_k}^{2\varkappa}h^{2} + \abs{ X_k }^{2\varkappa+1}h^2).
\end{align}

Now consider the second term in \eqref{eq:mom-bound-cross-term-key} :
\begin{eqnarray}\label{eq:mom-bound-cross-others}
&&   \mean{ \chi_{\tilde{\Omega}_{R,k}}(\omega)\abs{ X_k}^{2p-l}\abs{X_{k+1}-X_k}^l}  \\
&  \leq & K\mean{\chi_{\tilde{\Omega}_{R,k}}(\omega)\abs{ X_k} ^{2p-l} ( h^{l}|a(t_{k},X_{k})|^{l}+h^{l/2}\sum
_{r=1}^{m}|\sigma_{r}(t_{k},X_{k})|^{l}|\xi_{rk}|^{^{l}} ) }\notag\\
&  \leq & K\mathbb{E}\chi_{\tilde{\Omega}_{R,k}}(\omega)\left\vert
X_{k}\right\vert ^{2p-l}h^{l/2}\left[  1+h^{l/2}|X_k|^{l\varkappa}+|X_{k}|^{l(\varkappa+1)/2}\right]
,\notag
\end{eqnarray}
where we used the growth condition \eqref{eq:polynomial-growth-deduced} and the fact that
$\chi_{\tilde{\Omega}_{R,k}}(\omega)$ and $X_{k}$   are independent of $\xi_{rk}$.
 Then by  \eqref{eq:mom-bound-cross-term-key}, \eqref{eq:mom-bound-cross-conditionalexp}, and \eqref{eq:mom-bound-cross-others}, we have
\begin{eqnarray} \label{tml9}
&&  \mean{ \chi_{\tilde{\Omega}_{R,k+1}}(\omega)|X_{k+1}|^{2p} }\\
&  \leq&\mean{\chi_{\tilde{\Omega}_{R,k}}(\omega)|X_{k}|^{2p} +Kh\mathbb{E}\chi_{\tilde{\Omega}_{R,k}}(\omega)\abs{X_k}^{2p-2}\left[  1+|X_{k}|^{2}+h|X_{k}|^{2\varkappa+1} }\right] \notag\\
& & +K\sum_{l=3}^{2p}\mean{ \chi_{\tilde{\Omega}_{R,k}}(\omega)\abs{X_k} ^{2p-l}h^{l/2}\left[  1+h^{l/2}|X_k|^{l\varkappa}+|X_{k}|^{l(\varkappa+1)/2}\right] }
\notag\\
& \leq & \mean{\chi_{\tilde{\Omega}_{R,k}}(\omega)|X_{k}|^{2p}}
+Kh\mean{\chi_{\tilde{\Omega}_{R,k}}(\omega)\abs{X_k}^{2p }}+K\sum_{l=2}^{2p}\mathbb{E}\chi_{\tilde{\Omega}_{R,k}}(\omega)\left\vert
X_{k}\right\vert ^{2p-l}h^{l/2}\notag\\
&  &+Kh^{2}\mean{ \chi_{\tilde{\Omega}_{R,k}}(\omega)\abs{X_k}  ^{2p+2\varkappa-1} }+Kh\sum_{l=3}^{2p}\mathbb{E}\chi
_{\tilde{\Omega}_{R,k}}(\omega)\left\vert X_{k}\right\vert ^{2p+l(\varkappa
-1)/2}h^{l/2-1}.\notag
\end{eqnarray}

If  we choose
\begin{equation}\label{eq:range-rare-event}
R=R(h)= h^{-1/G(\varkappa)},\quad\mbox{where}\quad G(\varkappa)=\max(2\varkappa-1, \chi_{p>1}3(\varkappa-1)),
\end{equation}
we get, for $l=3,\ldots,2p,$
\begin{eqnarray*}
\chi_{\tilde{\Omega}_{R,k}}(\omega) \abs{ X_k}^{2p+2\varkappa-1}h  & \leq&  \chi_{\tilde{\Omega}_{R(h),k}}(\omega)\abs{X_k}^{2p}, \\
\chi_{\tilde{\Omega}_{R(h),k}}(\omega)\abs{X_k}^{2p+l(\varkappa-1)/2}h^{l/2-1} &\leq&\chi_{\tilde{\Omega}_{R(h),k}}(\omega)\abs{X_k}^{2p}.
\end{eqnarray*}
Thus we have for  \eqref{eq:mom-bound-cross-term-key},
\begin{align}
&  \mean{\chi_{\tilde{\Omega}_{R(h),k+1}}(\omega)|X_{k+1}|^{2p} }\\
&  \leq\mean{\chi_{\tilde{\Omega}_{R(h),k}}(\omega)|X_{k}|^{2p}}%
+Kh\mean{\chi_{\tilde{\Omega}_{R(h),k}}(\omega)\abs{X_k} ^{2p}  }+K\sum_{l=1}^{p}\mean{ \chi_{\tilde{\Omega}_{R(h),k}}(\omega
)\abs{X_ k}  ^{2(p-l)}}h^{l}\notag\\
&  \leq  \mean{\chi_{\tilde{\Omega}_{R(h),k}}(\omega)|X_{k}|^{2p}}
+Kh\mean{\chi_{\tilde{\Omega}_{R(h),k}}(\omega)|X_{k}|^{2p}}+Kh,\notag
\end{align}
where in the last line we have used Young's inequality. From here, we get \eqref{eq:mom-bound-nonrareevent} by
Gronwall's inequality.

It remains to estimate $\mean{\chi_{\tilde{\Lambda}_{R(h),k}}%
(\omega)|X_{k}|^{2p}}$. We recall that, see \cite[Section 3]{TreZhang13},
\begin{equation*}
\chi_{\tilde{\Lambda}_{R,k}} =\sum_{l=0}^{k}\chi_{\tilde{\Omega}_{R,l-1}}\chi_{|X_{l}|>R},
\end{equation*}
where we put $\chi_{\tilde{\Omega}_{R,-1}}=1.$ Then, using \eqref{eq:sine-tame-bound}, \eqref{eq:mom-bound-nonrareevent}, and H{\"o}lder's and Markov's
inequalities, we obtain
\begin{eqnarray}\mean{\chi_{\tilde{\Lambda}_{R(h),k}}%
(\omega)|X_{k}|^{2p}}
&\leq &\left(  \mean{\abs{\abs{X_0}+2k}^{4p}}\right)  ^{1/2}\sum_{l=0}^{k}\frac{\left(
\mean{\chi_{\tilde{\Omega}_{R(h),l-1}}|X_{l}|^{2(2p+1)G(\varkappa)}}\right)  ^{1/2}}{R(h)^{(2p+1)G(\varkappa)}}\notag\\
&\leq &K\left( \mean{\abs{\abs{X_0}+2k}^{4p}}\right)  ^{1/2}\left(  \mean{(1+\abs{X_0})^{2(2p+1)G(\varkappa)} }\right)  ^{1/2}kh^{2p+1}\notag\\
&\leq& K(1+\mean{\abs{X_0}^{4p+2(2p+1)G(\varkappa)} }).\notag
\end{eqnarray}
From   here  and  \eqref{eq:mom-bound-nonrareevent}, we obtain
\eqref{eq:mom-sine-tamed}  for integer $p\geq1$.
By Jensen's inequality,  \eqref{eq:mom-sine-tamed} holds for non-integer $p$ as well. 
\end{proof}

\subsection{One-step error }
The next lemma provides estimates for the one-step error of the balanced  Euler scheme
\eqref{eq:sine-tamed}.

\begin{lem}\label{lem:local-sine-tamed}
Assume that  \eqref{eq:sde-moment-bound-global-monotone} holds. Assume that the coefficients $a(t,x)$ and $\sigma_{r}(t,x)$ have continuous first-order
partial derivatives in $t$ and that these derivatives and the coefficients
satisfy inequalities of the form  \eqref{eq:polynomial-growth-deduced}. Then the scheme
\eqref{eq:sine-tamed}
 satisfies the inequalities
\eqref{eq:one-step-local-mean} and
\eqref{eq:one-step-local-ms}   with
$q_{1}=3/2$ and $q_{2}=1,$ respectively.

Moreover,  consider additive noise,   i.e., $\sigma_r(t,x)=\sigma_r(t)$.
If  the coefficient  $a(t,x)$  also has continuous first-order and second-order derivatives in $x$ and their derivatives satisfy the polynomial growth condition of the form
\eqref{eq:polynomial-growth-deduced},      then we have
   $q_{1}=2$ and $q_{2}=3/2$.
\end{lem}

The proof of this lemma is given below.
According to Theorem~\ref{thm:ms-convergence-local-global}, the following proposition can be readily deduced
from Lemmas~\ref{lem:mom-sine-tamed} and~\ref{lem:local-sine-tamed}.
\begin{thm}\label{prp:sine-tamed-order}
Under the assumptions of Lemmas \ref{lem:mom-sine-tamed} and~\ref{lem:local-sine-tamed}.
 the balanced Euler  scheme \eqref{eq:sine-tamed} has a mean-square convergence order
half, i.e., for it the inequality \eqref{eq:global-err-ms} holds with $q=1/2.$

For additive noise, we have  that the scheme \eqref{eq:sine-tamed} is of first-order convergence, i.e. $q=1$.
\end{thm}

We need the following lemma for the proof.

\begin{lem}[\cite{TreZhang13}] \label{lem:sde-nonlinear-holder}
Let a function $\varphi(t,x)$ have
continuous first-order partial derivative in t and that the derivative and the function
satisfy inequalities of the form (2.3). For $l\geq 1$ and  $s\geq t$, we have
\begin{equation} \label{eq:sde-nonlinear-holder}
\mean{\abs{\varphi(s,X_{t,x}(s))- \varphi(t, x)}^l} \leq K(1+\abs{x}^{2\varkappa l -l})[ (s-t)^{l/2}+(s-t)^l ].
\end{equation}
\end{lem}
The proof of  Lemma  \ref{lem:sde-nonlinear-holder} can be found in \cite[Appendix C]{TreZhang13}.
Now we prove Lemma \ref{lem:local-sine-tamed}, the order of accuracy for one-step error of the balanced Euler scheme  \eqref{eq:sine-tamed}.
\begin{proof}
 Now consider the one-step approximation of the SDE \eqref{eq:sde-ito-vec}, which
corresponds to the balanced method \eqref{eq:sine-tamed}: 
\begin{equation}\label{eq:sine-tamed-one}%
X=x+\sin(a(t,x)h)+\sin(\sum_{r=1}^{m}\sigma_{r}(t,x)\xi_{r}\sqrt{h}),
\end{equation}
and the one-step approximation corresponding to the explicit Euler scheme:
\begin{equation}\label{Euler_one}%
\tilde{X}=x+a(t,x)h+\sum_{r=1}^{m}\sigma_{r}(t,x)\xi_{r}\sqrt{h}.
\end{equation}

Step 1.
We start with analysis of the one-step error of the Euler scheme:
\[
\tilde{\rho}(t,x):=X_{t,x}(t+h)-\tilde{X}.
\]
By  Lemma \ref{lem:sde-nonlinear-holder}, we have
\begin{eqnarray}\label{new1-1}
\abs{\mean{\tilde{\rho}(t,x)}}&=&\abs{ \mathbb{E}\int_{t}^{t+h}%
(a(s,X_{t,x}(s))-a(t,x))ds} \\
&&\leq\mathbb{E}\int_{t}^{t+h}%
|a(s,X_{t,x}(s))-a(t,x)|ds\label{new1}\leq  Kh^{3/2}(1+|x|^{2\varkappa-1}).\notag
\end{eqnarray}
Also we have
\begin{eqnarray}\label{new3}
\mean{\abs{\tilde{\rho}}^{2p}(t,x)}  &  \leq & K\mathbb{E}\left\vert \int_{t}%
^{t+h}(a(s,X_{t,x}(s))-a(t,x))ds\right\vert ^{2p}\\
& & +K\sum_{r=1}^{q}\mathbb{E}\left\vert \int_{t}^{t+h}\left(  \sigma
_{r}(s,X_{t,x}(s))-\sigma_{r}(t,x)\right)  dw_{r}(s)\right\vert ^{2p}%
.\notag
\end{eqnarray}
By Lemma \ref{lem:sde-nonlinear-holder}, we
get for the first term in (\ref{new3}):
\begin{eqnarray}\label{eq:local-ms-error-drift}
\mean{\abs{\int_{t}^{t+h}(a(s,X_{t,x}(s))-a(t,x))ds }^{2p} }&\leq& Kh^{2p-1}\int_{t}^{t+h}\mean{\abs{a(s,X_{t,x}(s))-a(t,x)}^{2p} }\,ds \notag\\
&&\leq Kh^{3p}(1+|x|^{4p\varkappa-2p}).
\end{eqnarray}
Using the inequality for powers of Ito integrals from \cite[pp. 26]{GihSko-B72} and  Lemma \ref{lem:sde-nonlinear-holder},
we obtain
\begin{eqnarray}\label{new5}
&&\mean{\abs{\int_{t}^{t+h}\left(  \sigma_{r}(s,X_{t,x}(s))-\sigma_{r}(t,x)\right)  dw_{r}(s) } ^{2p} }\\
&\leq& Kh^{p-1}\int_{t}^{t+h}\mean{\abs{ \sigma_{r}(s,X_{t,x} (s))-\sigma_{r}(t,x) }^{2p} }\,ds\leq Kh^{2p}(1+|x|^{4p\varkappa
-2p}).\notag
\end{eqnarray}
It follows from (\ref{new3})-(\ref{new5}) that
\begin{equation}
\mean{\abs{\tilde{\rho}}^{2p}(t,x)}\leq Kh^{2p}(1+|x|^{4p\varkappa-2p}).
\label{new2}%
\end{equation}

Step 2. Now we compare the one-step approximations \eqref{eq:sine-tamed-one} of the balanced
scheme  \eqref{eq:sine-tamed} and (\ref{Euler_one}) of the Euler scheme:
\begin{equation}
X=x+\sin(a(t,x)h)+\sin(\sum_{r=1}^{m}\sigma_{r}(t,x)\xi_{r}\sqrt{h} )=\tilde{X}%
-\rho(t,x), \label{new7}%
\end{equation}
where
\[
\rho(t,x)=a(t,x)-\sin(a(t,x)) +\sum_{r=1}^{m}\sigma_{r}(t,x)\xi_{r}\sqrt{h}  -\sin(\sum_{r=1}^{m}\sigma_{r}(t,x)\xi_{r}\sqrt{h})
\]

By the symmetry of the normal distribution   and the asymmetry of sine function, we have
\[\mean{  \sin(\sum_{r=1}^{m}\sigma_{r}(t,x)\xi_{r}\sqrt{h})} =0,\]
and then by the inequality  \eqref{eq:basic-inequality-sinetamed} and the fact that $\abs{\sin(y)}\leq \abs{y}$, we have
\begin{eqnarray} \label{eq:local-err-mean-euler-sinetamed}
 \abs{\mean{\rho(t,x)} }
 &=&  \abs{ \mean{a(t,x)h -\sin(a(t,x)h }}
 \leq   \abs{  a(t,x)h} 2\abs{\sin( \theta a(t,x)h/2)}^2  \notag  \\
  &\leq& \abs{  a(t,x)h}  ^2 \leq  Kh^{2}(1+|x|^{2\varkappa}),
\end{eqnarray}
whence,  from (\ref{new7}) and (\ref{new1-1}), we obtain that \eqref{eq:sine-tamed-one}
satisfies \eqref{eq:one-step-local-mean} with $q_{1}=3/2.$

From the inequality  \eqref{eq:basic-inequality-sinetamed} and the fact that $\abs{\sin(y)}\leq \abs{y}$, we  can readily obtain
\begin{eqnarray} \label{eq:local-err-ms-euler-sinetamed}
\mean{\abs{\rho}^{2p}(t,x)}
&\leq& K\mean{ \abs{a(t,x)h - \sin(a(t,x)h }^{2p}}  \notag \\
&&  + K \mean{\abs{\sum_{r=1}^{m}\sigma_{r}(t,x)\xi_{r}\sqrt{h} -  \sin(\sum_{r=1}^{m}\sigma_{r}(t,x)\xi_{r}\sqrt{h})  }^{2p} }  \notag\\
&\leq&  K\abs{a(t,x)h}^{3p}   + K  h^{3p}\mean {\abs{ \sum_{r=1}^{m}\abs{\sigma_{r}(t,x)\xi_r} }  ^{6p}  }\leq Kh^{3p}%
(1+|x|^{3p(\varkappa+1)}), \notag
\end{eqnarray}
which together with (\ref{new7}) and (\ref{new2}) implies that \eqref{eq:sine-tamed-one} satisfies \eqref{eq:one-step-local-ms} with $q_{2}=1$. $\ \square$

For additive noise,  we assume that   the derivatives $\frac{\partial a }{\partial t}$ , $\frac{\partial a_i}{\partial x_j}$ and $\frac{\partial ^2 a_i}{\partial x_j \partial x_k}$
are continuous. Then we can write, by the Ito formula,
     \begin{eqnarray*}
  a(s,X_{t,x}(s)) - a(t,x) &=&\int_t^s \partial_t a(\theta,X_{t,x}(\theta))\,d\theta \\
  && +\sum_{r=1}^m \int_t^s  \Lambda_r a(\theta, X_{t,x}(\theta))\,dw_r(\theta)+\int_t^s L  a(\theta,X_{t,x}(\theta))\,d\theta .
  \end{eqnarray*}
   where $\Lambda_r= \sigma_r^\top \frac{\partial}{\partial x} = \sum_{i=1}\sigma_{i,r}\frac{\partial}{\partial x_i}$ and
$L=\frac{\partial}{\partial t} + a^\top \frac{\partial}{\partial x} + \frac{1}{2}\sum_{r=1}^m \sum_{i,j=1}^{d}\sigma_{i,r}\sigma_{j,r} \frac{\partial^2}{\partial x_i \partial x_j}$.
 Assuming that $\frac{\partial a }{\partial t}$ , $\frac{\partial a_i}{\partial x_j}$ and $\frac{\partial ^2 a_i}{\partial x_j \partial x_k}$ satisfies  the growth condition of the form \eqref{eq:polynomial-growth-deduced},   we can readily obtain
\begin{eqnarray} \label{eq:local-err-mean-euler-exact-additive}
\abs{\mean{\tilde{\rho}(t,x)}}
  &=&\abs{ \mean{\int_{t}^{t+h}(a(s,X_{t,x}(s))-a(t,x))ds}} \\
 &=& \abs{ \mean{\int_{t}^{t+h} \int_t^s \partial_t a(\theta,X_{t,x}(\theta) + L  a(\theta,X_{t,x}(\theta))\,d\theta\,ds }}  \leq Kh^{2}(1+|x|^{2\varkappa}).\notag
\end{eqnarray}

Also we have
\begin{eqnarray}\label{eq:local-err-ms-euler-exact}
\mean{\abs{\tilde{\rho}}^{2p}(t,x)}  & \leq & K\mathbb{E}\left\vert \int_{t}
^{t+h}(a(s,X_{t,x}(s))-a(t,x))ds\right\vert ^{2p} \\
& &+K\sum_{r=1}^{q}\mathbb{E}\left\vert \int_{t}^{t+h}\left(  \sigma
_{r}(s)-\sigma_{r}(t)\right)  dw_{r}(s)\right\vert ^{2p}
.\notag
\end{eqnarray}

By the inequality for powers of Ito integrals from \cite[pp. 26]{GihSko-B72} and smoothness of $\sigma_r(s)$, we have
\begin{gather}
\mean{\abs{ \int_{t}^{t+h}\left(  \sigma_{r}(s)-\sigma
_{r}(t)\right)  dw_{r}(s) } ^{2p} }
\leq Kh^{p-1}\int_{t}^{t+h}\mean{\abs{ \sigma_{r}(s)-\sigma_{r}(t)}^{2p} }ds\leq Kh^{3p}.\notag
\end{gather}
It then follows from here and  \eqref{eq:local-ms-error-drift} that
\begin{equation}\label{eq:local-err-ms-euler-exact-additive}
\mean{\abs{\tilde{\rho}}^{2p}(t,x)}\leq Kh^{3p}(1+|x|^{4p\varkappa-2p}).
\end{equation}

By \eqref{eq:local-err-mean-euler-sinetamed}  and \eqref{eq:local-err-mean-euler-exact-additive},
 we have  $q_1=2$.
 By \eqref{eq:local-err-ms-euler-exact-additive}  and \eqref{eq:local-err-ms-euler-sinetamed}, we have
 $q_2=3/2$.    Then by Theorem  \ref{thm:ms-convergence-local-global}, we have
 the scheme \eqref{eq:sine-tamed}  is first-order order convergence under additive noise.
\end{proof}

\begin{rem} 
	The  proofs for \eqref{eq:sine-tamed}  with $\mathcal{T}(\cdot)=\tanh(\cdot)$ are  similar. The   properties of the tame function we use in proofs are  a) $\abs{\sin(y)}$ is bounded, b) $\abs{\sin(y)}\leq \abs{y}$, c) $\sin(-y)=-\sin(y)$ and  also \eqref{eq:basic-inequality-sinetamed}. Note that the hyperbolic tangent has similar properties: 
	a) $\abs{\tanh(y)}\leq 1$, b) $\abs{\tanh(y)}\leq \abs{y}$, c) $\tanh(-y)=-\tanh(y)$  and also $\abs{y-\tanh(y)}\leq \tanh^2(\theta y)\abs{y}$ for some $0\leq \theta\leq 1$.
	
	Moreover, the hyperbolic tangent is  monotone while the sine function is not. The monotonicity may bring some advantages in practice even though the convergence rate will not change. For example,   we observe in  Example \ref{exm:com} that
	schemes with the hyperbolic tangent function allows 
	 larger time step sizes than  the sine tamed schemes do to obtain  accuracy and show convergence. See  Example \ref{exm:com} for more  comparison among solutions from these tamed schemes. 
\end{rem}

 \begin{rem}\label{rem:mod-tame-euler}
 Similar to the proofs above,  we can prove that
the following balanced scheme has  the same convergence  rate as
 the scheme \eqref{eq:sine-tamed} does:
 \begin{equation} \label{eq:sine-tamed-v1}
 X_{k+1} =X_k +  \mathcal{T}(a(t_k,X_k)h) +\sum_{k=1}^m \mathcal{T}(\sigma_r(t_k,X_k)\sqrt{h})\xi_k.
 \end{equation}
 Compared to \eqref{eq:sine-tamed}, the numerical solution is not bounded any more since the $\xi_k$'s can take values in the real line. However, numerical results (not resented) for both schemes (with different tame functions)  show similar error behaviors when both schemes are applied to Example \ref{exm:com}. 
 \end{rem}

    \section{The balanced Milstein scheme}

In this section, we prove that  the scheme \eqref{eq:sine-tamed-milstein} converges with strong order one.   We consider only the case when $\mathcal{T}(\cdot) =\sin(\cdot)$  as  the proof for $\mathcal{T}(\cdot)=\tanh(\cdot)$ is similar.  

\begin{lem}\label{lem:mom-sine-tamed-milstein}
Suppose Assumption \ref{asup:one-side-lip} holds with
sufficiently large $p_{0}$.  Assume the following polynomial growth for   $\Lambda_i\sigma_r(t,x)$:
\begin{equation}\label{eq:growth-derivatives}
 \sum_{i,r=1}^m \abs{\Lambda_i \sigma_r(t,x)}^2\leq K(1+ \abs{x}^{ 2\varkappa')}), \quad \varkappa\geq 0.
 \end{equation}
For all natural $N$ and all $k=0,\ldots,N$ the
following inequality holds for moments of the scheme \eqref{eq:sine-tamed}
\begin{equation}\label{eq:mom-sine-tamed-milstein}
\mathbb{E}|X_{k}|^{2p}\leq K(1+\mathbb{E}|X_{0}|^{2p\beta}),\ \ 2\leq
2p<\frac{2p_{0}}{G(\varkappa)}-1 ,
\end{equation}
where the constants $\beta\geq1$ and $K>0$ are independent of $h$ and $k$ and
\[G(\varkappa)=\max(2\varkappa-1,  2\varkappa'-2, \chi_{p>1}3(\varkappa-1)). \]
\end{lem}
 \begin{proof}
 The idea of the proof is similar to that for Lemma \ref{lem:mom-sine-tamed}. We thus present   the proof only with necessary details.

Again, the key to prove the boundedness of moments is to estimate of the growth of the solution under some events \begin{equation}
\tilde{\Omega}_{R,k}:=\{\omega:|X_{l}|\leq R(h),\ l=0,\ldots,k\},\ \label{event-milstein}%
\end{equation}
where $hR^{\varkappa}(h)<1$ such that
\begin{equation} \label{eq:mom-bound-nonrareevent-milstein}
  \mean{\chi_{\tilde{\Omega}_{R,k}}(\omega)\abs{X_k}^{2p} }\leq K(1+  \mean{\abs{X_0}^{2p} }).
\end{equation}

  We first prove the lemma for the integer $p\geq1.$ We have
\begin{eqnarray}\label{eq:mom-bound-cross-term-key-milstein}
&&   \mean{\chi_{\tilde{\Omega}_{R,k+1}}(\omega)\abs{X_{k+1}}^{2p} } \\
& \leq&\mean{\chi_{\tilde{\Omega}_{R,k}}(\omega)\abs{X_k}^{2p} }%
+K\sum_{l=3}^{2p}\mathbb{E}\chi_{\tilde{\Omega}_{R,k}}(\omega)\left\vert
X_{k}\right\vert ^{2p-l}|X_{k+1}-X_{k}|^{l}+\mean{\chi_{\tilde{\Omega}_{R,k}}(\omega)\abs{X_k}^{2p-2}A},\notag
\end{eqnarray}
where $\displaystyle A =  \chi_{\tilde{\Omega}_{R,k}}(\omega) \mean{  2p(X_{k},X_{k+1}-X_{k})+p(2p-1)|X_{k+1}-X_{k}|^{2} |\mathcal{F}_{t_k}}.$
 Compared to the proof of bounded moments for the scheme \eqref{eq:sine-tamed},
 it is essential to  provide
 a proper upper bound for  $A$.
 Similar to the proof of  the upper bound for $A$ in Lemma \ref{eq:mom-sine-tamed}, we have
 \begin{eqnarray*}
A  & = &2p\chi_{\tilde{\Omega}_{R,k}}\mean{(X_k, \sin(a(t_k,X_k)h)
   +\frac{2p-1}{2}\abs{X_{k+1}-X_k}^2 | \mathcal{F}_{t_k}}   \\
&=&2p\chi_{\tilde{\Omega}_{R,k}} \mean{(X_k, a(t_k,X_k)h)
    +\frac{2p-1}{2}(\abs{\sin(a(t_{k},X_{k})h)}^2+ \abs{  \sin(\sum_{r=1}^{m} \sigma_{r}(t_{k},X_{k})\xi_{rk}\sqrt{h})}^2)| \mathcal{F}_{t_k}}   \\
&&+ p(2p-1)\chi_{\tilde{\Omega}_{R,k}} \mean{ \abs{\sin(\sum_{i,r=1}^m\Lambda_i \sigma_r(t,X_k) I_{i,r,t_k})}^2| \mathcal{F}_{t_k}}  \\
&&+ p(2p-1)\chi_{\tilde{\Omega}_{R,k}} \mean{  2( \sin(\sum_{r=1}^{m} \sigma_{r}(t_{k},X_{k})\xi_{rk}\sqrt{h}),  \sin(\sum_{i,r=1}^m\Lambda_i \sigma_r(t,X_k) I_{i,r,t_k}))| \mathcal{F}_{t_k}}  \\
&&+ p(2p-1)\chi_{\tilde{\Omega}_{R,k}} \mean{  2( \sin(a(t_{k},X_{k})h),  \sin(\sum_{i,r=1}^m\Lambda_i \sigma_r(t,X_k) I_{i,r,t_k}))| \mathcal{F}_{t_k}}  \\
&&+ 2p\chi_{\tilde{\Omega}_{R,k}} \mean{(X_k,  a(t_k,X_k)h - \sin(a(t_k, X_k)h) \\
&&+  \sum_{i,r=1}^m\Lambda_i \sigma_r(t,X_k) I_{i,r,t_k}  -\sin(\sum_{i,r=1}^m\Lambda_i \sigma_r(t,X_k) I_{i,r,t_k}))| \mathcal{F}_{t_k}}
\\
&\leq&  \chi_{\tilde{\Omega}_{R,k}}(Kh\abs{X_k}^{2}   +     Kh^2\abs{X_k}^{2\varkappa}  +Kh^2 \abs{X_k}^{2\varkappa+1})\\
&&+ p(2p-1)\chi_{\tilde{\Omega}_{R,k}} \mean{ \abs{\sin(\sum_{i,r=1}^m\Lambda_i \sigma_r(t,X_k) I_{i,r,t_k})}^2| \mathcal{F}_{t_k}} \\
&&  + p(2p-1)\chi_{\tilde{\Omega}_{R,k}} \mean{  2( \sin(\sum_{r=1}^{m} \sigma_{r}(t_{k},X_{k})\xi_{rk}\sqrt{h}),  \sin(\sum_{i,r=1}^m\Lambda_i \sigma_r(t,X_k) I_{i,r,t_k}))| \mathcal{F}_{t_k}}  \\
&&+p(2p-1)\chi_{\tilde{\Omega}_{R,k}} \mean{  2( \sin(a(t_{k},X_{k})h),  \sin(\sum_{i,r=1}^m\Lambda_i \sigma_r(t,X_k) I_{i,r,t_k}))| \mathcal{F}_{t_k}}\\
&&+ 2p\chi_{\tilde{\Omega}_{R,k}} \mean{   \sum_{i,r=1}^m\Lambda_i \sigma_r(t,X_k) I_{i,r,t_k}  -\sin(\sum_{i,r=1}^m\Lambda_i \sigma_r(t,X_k) I_{i,r,t_k}))| \mathcal{F}_{t_k}}.
\end{eqnarray*}

 By the symmetry of the normal distribution  and the asymmetry of the sine function, we have
 \begin{equation}
 \chi_{\tilde{\Omega}_{R,k}} \mean{  2( \sin(\sum_{r=1}^{m} \sigma_{r}(t_{k},X_{k})\xi_{rk}\sqrt{h}),  \sin(\sum_{i,r=1}^m\Lambda_i \sigma_r(t,X_k) I_{i,r,t_k}))| \mathcal{F}_{t_k}}=0.
 \end{equation}
 From  the fact that $\abs{\sin(y)}\leq \abs{y}$ and the inequality  \eqref{eq:basic-inequality-sinetamed}, we  can readily obtain
 \begin{eqnarray}
 &&\chi_{\tilde{\Omega}_{R,k}} \mean{ \abs{\sin(\sum_{i,r=1}^m\Lambda_i \sigma_r(t,X_k) I_{i,r,t_k})}^2| \mathcal{F}_{t_k}} \notag\\
   & & \leq  K\chi_{\tilde{\Omega}_{R,k}}   (\sum_{i,r=1}^m \abs{\Lambda_i \sigma_r(t,X_k)}^2 \mean{ I_{i,r,t_k}^2 |\mathcal{F}_{t_k} }, \\
&& \chi_{\tilde{\Omega}_{R,k}} \mean{  2( \sin(a(t_{k},X_{k})h),  \sin(\sum_{i,r=1}^m\Lambda_i \sigma_r(t,X_k) I_{i,r,t_k}))| \mathcal{F}_{t_k}} \notag\\
&& \quad  \leq K \chi_{\tilde{\Omega}_{R,k}}  \abs{a(t_k,X_k)h}    \sum_{i,r=1}^m \abs{\Lambda_i \sigma_r(t,X_k)} \mean{ \abs{I_{i,r,t_k} }| \mathcal{F}_{t_k}},\\
&& \chi_{\tilde{\Omega}_{R,k}} \mean{   \sum_{i,r=1}^m\Lambda_i \sigma_r(t,X_k) I_{i,r,t_k}  -\sin(\sum_{i,r=1}^m\Lambda_i \sigma_r(t_k,X_k) I_{i,r,t_k}))| \mathcal{F}_{t_k}}\notag\\
&&   \quad\leq   K\chi_{\tilde{\Omega}_{R,k}}  \sum_{i,r=1}^m \abs{\Lambda_i \sigma_r(t,X_k)}^2\mean{  I_{i,r,t_k}^2 | \mathcal{F}_{t_k}}.
 \end{eqnarray}
 Recall the fact that  $(\mean{  \abs{I_{i,r,t_k}} | \mathcal{F}_{t_k}})^2 \leq  \mean{  I_{i,r,t_k}^2 | \mathcal{F}_{t_k}}  \leq K h^2$, see e.g. \cite[pp. 20]{MilTre-B04} and we then have
  \begin{eqnarray*}
A  &\leq&  \chi_{\tilde{\Omega}_{R,k}}(Kh\abs{X_k}^{2}   +     Kh^2\abs{X_k}^{2\varkappa}  +Kh^2 \abs{X_k}^{2\varkappa+1} )\\
&&+  \chi_{\tilde{\Omega}_{R,k}}(Kh^2\abs{a(t_k,X_k)}^2+Kh^2\sum_{i,r=1}^m \abs{\Lambda_i \sigma_r(t,X_k)}^2 ).
\end{eqnarray*}
 By the growth conditions \eqref{eq:polynomial-growth-deduced}  and \eqref{eq:growth-derivatives} , we have
   \begin{eqnarray*}
A  &\leq&  \chi_{\tilde{\Omega}_{R,k}}(Kh\abs{X_k}^{2}   +     Kh^2\abs{X_k}^{2\varkappa+1} +Kh^2\sum_{i,r=1}^m \abs{\Lambda_i \sigma_r(t,X_k)}^2 )\\
&\leq &  \chi_{\tilde{\Omega}_{R,k}}(Kh\abs{X_k}^{2}   +     Kh^2\abs{X_k}^{2\varkappa+1} +Kh^2 \abs{X_k}^{2\varkappa'} ).
\end{eqnarray*}

 Similar to the proof of Lemma \eqref{eq:mom-sine-tamed}, we set $R(h)= h^{-1/G(\varkappa)}$ where $G(\varkappa)= \max(2\varkappa-1, 2\varkappa'-2, \chi_{p>1}3(\varkappa-1))$, we can show that
 \begin{equation}
 \mean{\abs{X_k}^{2p}}\leq K (1+ \mean{ \abs{X_0 }^{4p +2(2p+1)G(\varkappa)}}), \quad 2\leq 2p <\frac{2p_0}{G(\varkappa)}- 1.
 \end{equation}
 This ends the proof.
 \end{proof}

 \subsection{One-step error}

 Now consider the one-step approximation of the SDE \eqref{eq:sde-ito-vec}, which
corresponds to the balanced method \eqref{eq:sine-tamed-milstein}:
\begin{equation} \label{eq:sine-tamed-milstein-one}%
X=x+\sin (a(t,x)h)+\sin(\sum_{r=1}^{m}\sigma_{r}(t,x)\xi_{rk}\sqrt{h}) +\sin(\sum_{i,r=1}^m\Lambda_i\sigma_r(t,x) I_{i,r,t}),
\end{equation}
where
 $ \displaystyle I_{i,r,t} = \int_{t} ^{t+h} \int_{t} ^{s} \,dw_i\,dw_r$
and the one-step approximation corresponding to the Milstein scheme:
\begin{equation}\label{eq:milstein-one-step}%
\tilde{X}=x+a(t,x)h+\sum_{r=1}^{m}\sigma_{r}(t,x)\xi_{r}\sqrt{h}+\sum_{i,r=1}^m\Lambda_i\sigma_r(t,x) I_{i,r,t}.
\end{equation}

\begin{lem}\label{lem:local-sine-tamed-milstein}
Assume that  \eqref{eq:sde-moment-bound-global-monotone} holds and that the coefficients $a(t,x)$ and $\sigma_{r}(t,x)$ have continuous first-order
partial derivatives in $t$ and  up to third-order derivatives in $x$. Also assume that these derivatives and the coefficients
satisfy inequalities of the following form:
\begin{equation} \label{eq:polynomial-growth-derivatives}
\abs{\partial_t a(t,x)},\abs{\partial_t \sigma_r(t,x)}, \abs{L a(t,x)},  \abs{L\sigma_r(t,x)} \abs{\partial_t \Lambda_i\sigma_r(t,x)},\abs{L\Lambda_i\sigma_r(t,x)}  \leq K(1+\abs{x}^{\varkappa''}),
\end{equation}
where $\varkappa''\geq0$.    Then the scheme
\eqref{eq:sine-tamed-milstein}
 satisfies the inequalities
\eqref{eq:one-step-local-mean} and
\eqref{eq:one-step-local-ms}   with
$q_{1}= 2$ and $q_{2}=3/2,$ respectively.
\end{lem}

The proof of this lemma is given below.
According to Theorem~\ref{thm:ms-convergence-local-global}, the following proposition can be readily deduced
from Lemmas~\ref{lem:mom-sine-tamed-milstein} and~\ref{lem:local-sine-tamed-milstein}.
\begin{thm}\label{prp:sine-tamed-order-milstein}
Under the assumptions of Lemmas \ref{lem:mom-sine-tamed-milstein} and~\ref{lem:local-sine-tamed-milstein}.
 the balanced Milstein  scheme \eqref{eq:sine-tamed-milstein} has a mean-square convergence order
one, i.e.,  the inequality \eqref{eq:global-err-ms} holds with $q=1.$
\end{thm}

\begin{proof}
Step 1.
We start with the analysis of the one-step error of the Milstein scheme:
\[
\tilde{\rho}(t,x):=X_{t,x}(t+h)-\tilde{X}.
\]

By the Ito  formula,   we obtain
 \begin{eqnarray} \label{eq:local-err-mean-milstein-exact}
&&\abs{\mean{\tilde{\rho}(t,x)}}  \notag\\
  &=&\abs{ \mean{\int_{t}^{t+h}(a(s,X_{t,x}(s))-a(t,x))ds+\sum_{i,r=1}^m \int_t^{t+h}\int_t^s [\Lambda_i\sigma_r(\theta,X_{t,x}(\theta))-\Lambda_i\sigma_r(t,x)]  \,dw_i\,dw_r}} \notag\\
 &\leq& \abs{ \mean{\int_{t}^{t+h} \int_t^s \partial_ta(\theta,X_{t,x}(\theta))    + L a(\theta,X_{t,x}(\theta))\,d\theta\,ds  }}  \notag \\
 &&+\abs{\mean{\sum_{r=1}^m \int_t^{t+h}\int_t^s  \int_t^{\theta} \partial_t(\Lambda_r\sigma_r(\theta_1,X_{t,x}(\theta_1)))+ L\Lambda_r\sigma_r(\theta_1,X_{t,x}  (\theta_1))\,d\theta_1\,dw_r(\theta)\,dw_r(s)} }  \notag\\
  &\leq& Kh^{2}(1+|x|^{ \varkappa''}),
\end{eqnarray}
where we have used  H{\"o}lder's inequality and the growth condition \eqref{eq:polynomial-growth-derivatives}.

For the mean-square one-step error, we have
\begin{eqnarray}\label{eq:local-err-ms-milstein-exact}
\mean{\abs{\tilde{\rho}}^{2p}(t,x)}  & \leq & K\mean{\abs{\int_{t}
^{t+h}(a(s,X_{t,x}(s))-a(t,x))ds }^{2p}  }\\
&&+K\sum_{r=1}^{q}\mean{\abs{ \int_{t}^{t+h} (  \sigma
_{r}(s,X_{t,x}(s))-\sigma_{r}(t,x)  dw_{r}(s) -\sum_{i,r=1}^m\Lambda_i\sigma_r(t,x) I_{i,r,t}} ^{2p} }.\notag
\end{eqnarray}
By the Ito  formula, using the inequality for powers of Ito integrals from \cite[pp. 26]{GihSko-B72}, we obtain
\begin{eqnarray*}
&&\mean{\abs{ \int_{t}^{t+h}(\sigma_r(s,X_{t,x}(s))-\sigma_r(t,x))\,dw_r(s)-\sum_{i,r=1}^m\Lambda_i\sigma_r(t,x) \int_t^{t+h}\int_t^s \,dw_i\,dw_r}^{2p} } \notag  \\
&\leq&  Kh^{p-1}\int_{t}^{t+h}\mean{\abs{\sum_{i,r=1}^m \int_t^s[  \Lambda_i \sigma_r(\theta, X_{t,x}(\theta)) -\Lambda_i\sigma_r(t,x) ] \,dw_i(\theta)} }^{2p}\,ds   \notag    \\
&& + K h^{p-1} \int_{t}^{t+h}\mean{ \abs{\int_t^s  \partial_t\sigma_r(\theta,X_{t,x}(\theta))+L  \sigma_r(\theta,X_{t,x}(\theta))\,d\theta }^{2p}}\,ds   \\
&\leq&  Kh^{p-1}\int_{t}^{t+h}  (t-s)^{p-1} \sum_{i,r=1}^m \int_t^s  \mean{\abs{\Lambda_i \sigma_r(\theta, X_{t,x}(\theta))-\Lambda_i\sigma_r(t,x)}^{2p} } \,d \theta\,ds   \notag    \\
&& + K h^{p-1} \int_{t}^{t+h}\mean{ \abs{\int_t^s  \partial_t\sigma_r(\theta,X_{t,x}(\theta))+L  \sigma_r(\theta,X_{t,x}(\theta))\,d\theta }^{2p}}\,ds,
\end{eqnarray*}
which can be further estimated as, using the Ito  formula for $\Lambda_i\sigma_r$ and H{\"o}lder's inequality,
\begin{eqnarray}\label{eq:local-ms-error-diffusion-mil}
&&\mean{\abs{ \int_{t}^{t+h}(\sigma_r(s,X_{t,x}(s))-\sigma_r(t,x))\,dw_r(s)-\sum_{i,r=1}^m\Lambda_i\sigma_r(t,x) \int_t^{t+h}\int_t^s \,dw_i\,dw_r}^{2p} } \notag  \\
&\leq&  Kh^{p-1}\int_{t}^{t+h}  (t-s)^{2p} ( 1+\abs{x}^{2p\varkappa''-2p})  \,ds     +
K h^{3p}  (1+\abs{x}^{2p\varkappa''-2p}  )\notag   \\
& \leq  & K h^{3p}  (1+\abs{x}^{2p\varkappa''-2p}  ),
\end{eqnarray}
where we have used  the growth condition \eqref{eq:polynomial-growth-derivatives}.
Similarly, we have
\begin{eqnarray}\label{eq:local-ms-error-drift-milstein}
\mean{\abs{\int_{t}^{t+h}(a(s,X_{t,x}(s))-a(t,x))ds }^{2p} }&\leq& Kh^{2p-1}\int_{t}^{t+h}\mean{\abs{a(s,X_{t,x}(s))-a(t,x)}^{2p} }\,ds\notag\\
&&\leq Kh^{3p}(1+|x|^{2p\varkappa''-2p}).
\end{eqnarray}
By  \eqref{eq:local-ms-error-diffusion-mil} and \eqref{eq:local-ms-error-drift-milstein},
 we obtain
 \begin{equation}  \label{eq:local-ms-error-milstein-exact}
\mean{\abs{\tilde{\rho}}^{2p}(t,x)}   \leq Kh^{3p}(1+|x|^{2p\varkappa''-2p}).
 \end{equation}

Step 2. Now we compare the one-step approximations \eqref{eq:sine-tamed-milstein-one} of the balanced
scheme \eqref{eq:sine-tamed-milstein} and   \eqref{eq:milstein-one-step} of the Milstein scheme.
 Define
\begin{eqnarray}
\rho(t,x) &=& a(t,x)h -\sin(a(t,x)h)+\sum_{r=1}^{m}\sigma_{r}(t,x)\xi_{r}\sqrt{h} -\sin(\sum_{r=1}^{m}\sigma_{r}(t,x)\xi_{r}\sqrt{h} ) \notag\\
&& + \sum_{i,r=1}^m  \Lambda_i \sigma_r(t,x) I_{i,r,t}-\sin(\sum_{i,r=1}^m  \Lambda_i \sigma_r(t,x) I_{i,r,t}).
\end{eqnarray}

By the symmetry of the normal distribution and the asymmetry of  the sine function, we have
\begin{eqnarray} \label{eq:local-err-mean-milstein-sinetamed}
 \abs{\mathbb{E}\rho(t,x)}&\leq &  \abs{\mean{  a(t,x)h-\sin(a(t,x)h)  }}   \notag\\
 & & +  \abs{ \mean{  \sum_{i,r=1}^m  \Lambda_i \sigma_r(t,x) I_{i,r,t}-\sin(\sum_{i,r=1}^m  \Lambda_i \sigma_r(t,x) I_{i,r,t})}}
 \notag \\
   &\leq&  \abs{a(t,x)h}^2 +  \mean{\abs{\sum_{i,r=1}^m  \Lambda_i \sigma_r(t,x) I_{i,r,t}}^2}\notag  \\
      &\leq&  \abs{a(t,x)h}^2 + K \sum_{i,r=1}^m  \abs{\Lambda_i \sigma_r(t,x)}^2 \mean{ I_{i,r,t}^2}\notag  \\
   &\leq & Kh^{2}(1+|x|^{\max(2\varkappa,2\varkappa')}),
\end{eqnarray}
where  we have used the  inequality \eqref{eq:basic-inequality-sinetamed}, the polynomial growth conditions \eqref{eq:polynomial-growth-deduced} and  \eqref{eq:growth-derivatives},  and $\mean{I_{i,r,t}^2}\leq K h^2$.
From here  and \eqref{eq:local-err-mean-milstein-exact}, we have the one-step  approximation of \eqref{eq:sine-tamed-milstein}, \eqref{eq:sine-tamed-milstein-one}, satisfies \eqref{eq:one-step-local-mean} with $q_{1}=2$.

From the inequality \eqref{eq:basic-inequality-sinetamed} and $\abs{\sin(y)}\leq\abs{y}$,  we  can readily obtain
\begin{eqnarray} \label{eq:local-err-ms-milstein-sinetamed}
\mean{\abs{\rho}^{2p}(t,x)}&\leq&
K  \mean{\abs{a(t,x)h -\sin(a(t,x)h)}^{2p} }  \notag\\
&&+ K\mean{\abs{\sum_{r=1}^{m}\sigma_{r}(t,x)\xi_{r}\sqrt{h} -\sin(\sum_{r=1}^{m}\sigma_{r}(t,x)\xi_{r}\sqrt{h}}^{2p}} \notag \\
&&+ K\mean{ \abs{ \sum_{i,r=1}^m  \Lambda_i \sigma_r(t,x) I_{i,r,t_k}-\sin(\sum_{i,r=1}^m  \Lambda_i \sigma_r(t,x) I_{i,r,t_k}) }^{2p}}  \notag \\
&\leq&   K \abs{a(t,x)h }^{3p}  + K \mean{\abs{ \sum_{r=1}^m\sigma_r(t,x)\xi_r\sqrt{h}}^{6p}}+  K\mean{\abs{  \sum_{i,r=1}^m  \Lambda_i \sigma_r(t,x) I_{i,r,t_k} }^{3p}}  \notag \\
&\leq&   K h^{3p}(\abs{a(t,x) }^{3p}  +  \sum_{r=1}^m\abs{\sigma_r(t,x)}^{6p} +   \sum_{i,r=1}^m \abs{   \Lambda_i \sigma_r(t,x) }^{3p}).
\end{eqnarray}
Thus by  the polynomial growth conditions \eqref{eq:polynomial-growth-deduced} and \eqref{eq:growth-derivatives}, we obtain
\begin{eqnarray}
\mean{\abs{\rho}^{2p}(t,x)}&\leq &  Kh^{3p}(1+\abs{x}^{\max(3p(\varkappa+1),3p\varkappa')} ), \notag
\end{eqnarray}which together with \eqref{eq:local-ms-error-milstein-exact} implies that  the one-step  approximation of \eqref{eq:sine-tamed-milstein}, \eqref{eq:sine-tamed-milstein-one}, satisfies \eqref{eq:one-step-local-ms} with $q_{2}=3/2$.
 \end{proof}

 \begin{rem}
  Similar to the proof above,  we can prove that
 the following balanced scheme has  the same convergence  rate as
   the scheme \eqref{eq:sine-tamed-milstein} does:
  \begin{eqnarray} \label{eq:sine-tamed-milstein-v1}
    X_{k+1}  &=& X_{k}+   \mathcal{T}(a(t_k,X_{k}) h )+\sum_{r=1}^{m}\mathcal{T}(\sigma_{r}(t_{k},X_{k})\sqrt{h} )    \xi_{rk} \\
    && + \sum_{i,r=1}^m\mathcal{T}(\Lambda_i \sigma_r(t,X_k)h)\frac{I_{i,r,t_k}}{h}. \notag
 \end{eqnarray}
  Similar to \eqref{eq:sine-tamed-v1}, the numerical 
		solution is not bounded since the $\xi_k$'s can take values in the real line. Numerical results (not resented) for both schemes (with different tame functions)  show similar error behaviors when both schemes are applied to Example \ref{exm:com}. 
 \end{rem}

\section{Numerical results}
In this section, we present some numerical results for our proposed schemes and test the mean-square convergence orders.
To compute the mean-square error, we run $M$ independent trajectories
$X^{(i)}(t),$ $X_{k}^{(i)}$:
\begin{equation}
\left(  E\left[  X(T)-X_{N}\right]  ^{2}\right)  ^{1/2}\doteq\left(  \frac
{1}{M}\sum_{i=1}^{M}[X^{(i)}(T)-X_{N}^{(i)}]^{2}\right)  ^{1/2},
\label{experr}
\end{equation}
where   $T=5$ and $M=10^{4}.$ The reference solution (we don't have $X(T)$ and thus we need a good approximation of $X(T)$) was computed by the
mid-point method \cite{MilTre-B04,TreZhang13}  with small time step $h=10^{-5}$:
\begin{eqnarray}\label{eq:midpoint}
X_{k+1}  &  = & X_{k}+a(t_{k+1/2}, (X_{k}+ X_{k+1})/2)h+\sum_{r=1}^{m}\sigma_{r}(t_{k+1/2},  (X_{k}+ X_{k+1})/2)\left(  \zeta_{rh}\right)  _{k}\sqrt{h}\  \notag \\
& & -1/2\sum_{r=1}^{m}\sum_{j=1}^{d}\frac{\partial\sigma_{r}}{\partial
x^{j}}(t_{k+1/2}, (X_{k}+ X_{k+1})/2)\sigma_{j,r}(t_{k+\lambda}, (X_{k}+ X_{k+1})/2)h, 
\end{eqnarray}
where    $t_{k+1/2}=t_{k}+  h/2$ and $\left(
\zeta_{rh}\right)  _{k}$ are i.i.d. random variables so that
\begin{equation}
\zeta_{h}=\left\{
\begin{array}
[c]{c}%
\xi,\;|\xi|\leq A_{h},\\
A_{h},\;\xi>A_{h},\\
-A_{h},\;\xi<-A_{h},
\end{array}
\right.  \label{Fin61}%
\end{equation}
with $\xi$ $\sim$ $\mathcal{N}(0,1)$ and $A_{h}=\sqrt{2l|\ln h|}$ with
$l\geq1.$ Here  we took $l = 2$. Newton's method was used to solve the nonlinear algebraic equations
at each step of the implicit scheme.
It was verified that using  other schemes for simulating a reference solution with enough resolution does not
affect the convergence orders.
The experiments were performed
using Matlab R2014b (64 bit) and we used
 the Matlab command  \textit{rng(100,`twister')} to generate random numbers.

\begin{exm}
\label{exm:com} Consider  the following Stratonovich SDE of the form:
\begin{equation}
dX=(1-X^{5})\,dt+\sigma X^2\circ\,dw_{1},\quad X(0)=0. \label{Example1}%
\end{equation}

\end{exm}

 In  Ito's sense, the drift of the equation becomes $a(t,x)=1-x^{5}%
+\sigma^2 x^3$.
In the following tables, we present some numerical results of
the mid-point scheme and our explicit schemes, where the statistical errors with  the $95\%$ confidence level can be ignored.

The mid-point is  of order one in the mean-square sense as the coefficients of the noise satisfies the commutative conditions \eqref{eq:commutative-condition}, see e.g. \cite{TreZhang13}. The  Milstein scheme \eqref{eq:sine-tamed-milstein} is
of   order one  and the tamed Euler scheme \eqref{eq:sine-tamed}  is  of order half as  predicted from our proved convergence orders.  
	In this example, we have commutative noises and use  then  the scheme \eqref{eq:sine-tamed-milstein-com} instead of \eqref{eq:sine-tamed-milstein} in computation.	
	We first take 
	$\sigma=0.5$ and test our tamed schemes for Equation \eqref{Example1} up to $T=50$. In Table \ref{tab:mean-square-sigma-small}, the two tamed Euler schemes \eqref{eq:sine-tamed} with 
different tame functions yield similar mean-square errors and  convergence orders.
The two tamed Milstein schemes \eqref{eq:sine-tamed-milstein} with $\mathcal{T}(\cdot)=\tanh(\cdot)$ and $\mathcal{T}(\cdot)=\sin(\cdot)$ also have similar errors and convergence orders.  The errors from the tamed Euler (Milstein) schemes with the tame function $\tanh(\cdot)$  are of the same magnitudes
as the sine tamed Euler (Milstein) schemes   if the same time step size $h$ is used.

 \begin{table}[h]
 	\caption{Mean-square errors of mid-point scheme and our explicit schemes for Example \ref{exm:com} when  $\sigma=0.5$ at $T=50$.}  \label{tab:mean-square-sigma-small} \smallskip
 	\centering
 	\scalebox{0.80}{ 	
 		\begin{tabular}  {|c ||c |c  ||c  |c||c  |c||c|c||c|c|}   \hline
 			$h$  &    \eqref{eq:midpoint}  & rate    &   \eqref{eq:sine-tamed-milstein}-$\tanh$  &  rate  &    \eqref{eq:sine-tamed}-$\tanh$ &  rate   &   \eqref{eq:sine-tamed-milstein}-sine  &  rate  &    \eqref{eq:sine-tamed}-sine &  rate    \\ \hline \hline
 			2e-2 & 5.7482e-3 & --     &  1.2514e-2  &  --     & 2.9237e-2 & --   & 1.3010e-2  & --   & 2.9994e-2 &  --  \\ \hline
 			1e-2 & 2.8678e-3 & 1.00   &  5.6681e-3  & 1.14    & 1.8266e-2 & 0.68 & 5.7848e-3  & 1.17 & 1.8525e-2 & 0.70 \\ \hline
 			5e-3 & 1.4510e-3 & 0.98   &  2.7578e-3  & 1.04    & 1.2299e-2 & 0.57 & 2.8129e-3  & 1.04 & 1.2382e-2 & 0.58 \\ \hline
 			2e-3 & 5.7308e-4 & 1.01   &  1.0720e-3  & 1.03    & 7.4287e-3 & 0.55 & 1.0908e-3  & 1.03 & 7.4526e-3 & 0.55 \\ \hline
 			1e-3 & 2.8997e-4 & 0.98   &  5.2751e-4  & 1.02    & 5.2495e-3 & 0.50 & 5.3647e-4  & 1.02 & 5.2568e-3 & 0.50 \\ \hline
 			5e-4 & 1.4223e-4 & 1.03   &  2.6535e-4  & 0.99    & 3.7300e-3 & 0.49 & 2.6994e-4  & 0.99 & 3.7330e-3 & 0.49 \\ \hline
 		\end{tabular}  }
 	\end{table}
 
  We  now test  these   schemes when $\sigma=1$ and $T=50$ and present the results for tame schemes with the hyperbolic tangent function in Table \ref{tab:mean-square-sigma-large}.  The tamed schemes with $\tanh(\cdot)$  work well even with large time step sizes\footnote{In practice, it is useful to have an estimate  of the size of small $h$ to guarantee accuracy. However, we did not succeed in establishing such an estimate for balanced/tamed schemes in theory  or find    any efforts in this direction in literature.}. The schemes with  
    with the sine tame function can not lead to reasonable accuracy with  time step sizes larger than $0.002$.  However, with smaller time step sizes, we can  obtain very satisfactory accuracy from our schemes and the accuracy is very similar to that from  the mid-point scheme \eqref{eq:midpoint}. For example, when $h=0.002$, the mean-square error of the sine tamed Milstein scheme  is $6.2773\times 10^{-3}$ and the mean-square error of the 
    sine tamed Euler  scheme  is $3.2381\times 10^{-2}$. Moreover, the convergence rates of the sine tamed schemes are the same as predicted (numerical results not represented). 

  \begin{table}[h]
  	\caption{Mean-square errors of mid-point scheme and our explicit schemes for Example \ref{exm:com} when  $\sigma=1$ at $T=50$.
  	} \label{tab:mean-square-sigma-large} \smallskip
  	\centering
 	
  		\begin{tabular}  {|c||c |c  ||c  |c||c  |c||c|c||c|c|}   \hline
  			$h$  &    \eqref{eq:midpoint}  & rate    &   \eqref{eq:sine-tamed-milstein}-$\tanh$  &  rate  &    \eqref{eq:sine-tamed}-$\tanh$ &  rate      \\ \hline \hline 
1e-1& 8.5591e-2  & --    &  3.4851e-1 &  --    &  3.9115e-1 & --               \\ \hline 
5e-2& 5.2338e-2  &0.71   &  1.8817e-1 &  0.89  &  2.4587e-1 & 0.67              \\ \hline
2e-2& 2.2158e-2  &0.94   &  7.3670e-2 &  1.02  &  1.2392e-1 & 0.75             \\ \hline
1e-2& 1.1293e-2  &0.97   &  3.3404e-2 &  1.14  &  7.6968e-2 & 0.69             \\ \hline
5e-3& 6.0575e-3  &0.90   &  1.5098e-2 &  1.15  &  5.2444e-2 & 0.55             \\ \hline
2e-3& 2.4991e-3  &0.97   &  6.1777e-3 &  0.98  &  3.1631e-2 & 0.55             \\ \hline
1e-3& 1.1832e-3  &1.08   &  3.1113e-3 &  0.99  &  2.2456e-2 & 0.49             \\ \hline
5e-4& 5.8179e-4  &1.02   &  1.3976e-3 &  1.15  &  1.5892e-2 & 0.50             \\ \hline
 \end{tabular}   
  	\end{table}

 To understand why   schemes with the  sine tame function are not working well, let us  look at some paths of solutions to \eqref{Example1} with different initial values.   In Figure 1, we plot  solution paths up to $T=5$ using these five  schemes with 
$h=10^{-3}$ when $X(0)=0$ and $X(0)=100$. In the computation, we skip $5\times 10^5$ random numbers. 
When  $X(0)=0$, all the schemes  capture    the same magnitudes. 
While for  $X(0)=100$, we observe that the  tamed   schemes using the sine function can not obtain  correct magnitudes of the solution for $t\leq 4.5$. Moreover, the sine tamed Milstein scheme eventually reach zero (the correct magnitude) when $t$ is around 5. This test suggests that the sine tamed schemes may need smaller time step sizes to show convergence and obtain moderate accuracy.  

\begin{figure}[!h]
	\caption{Solution paths from the midpoint scheme \eqref{eq:midpoint} and tamed schemes \eqref{eq:sine-tamed} and \eqref{eq:sine-tamed-milstein}. The time step size $h$ is $10^{-3}$. Left: initial value $X(0)=0$. Right: initial value $X(0)=100$.}
	\centering
	\includegraphics[scale=0.46]{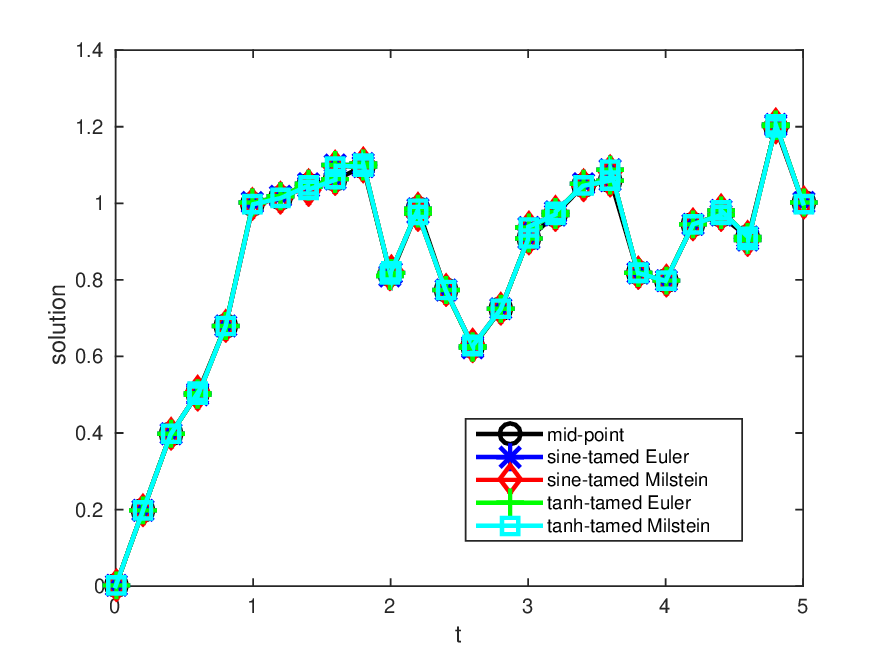}\hskip  -1 pt
	\includegraphics[scale=0.46]{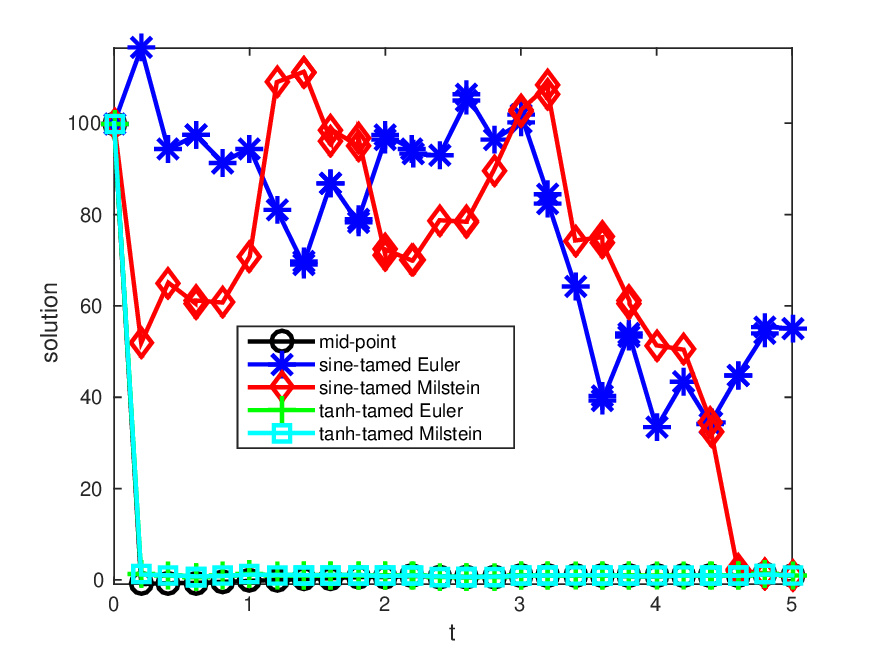}
\end{figure}

 In summary, we conclude from this example that tamed schemes can preserve the convergence orders (half-order and first-order). Though the tamed schemes are convergent,   requirements on time step sizes   are different. 
Numerical results show that schemes with 	
	  the hyperbolic tangent function allows larger time steps   than those with the sine tame function.  

\section*{Acknowledgement}
The work of the first author  was partially supported by a start-up fund from WPI. 
The work of the second author was supported by the NSFC grant 11571224.

  
  \def\cprime{$'$} \def\polhk#1{\setbox0=\hbox{#1}{\ooalign{\hidewidth
  \lower1.5ex\hbox{`}\hidewidth\crcr\unhbox0}}} \def\cprime{$'$}
 
\end{document}